\documentclass[12pt]{article}

\usepackage{amsmath,amssymb,amsthm}
\usepackage[utf8]{inputenc}

\usepackage{tikz}

\newcommand\NN{\mathbb{N}}

\theoremstyle{plain}
\newtheorem{theorem}{Theorem}[section]
\newtheorem{proposition}{Proposition}[section]

\newtheorem{lemma}{Lemma}[section]
\newtheorem{corollary}{Corollary}[section]

\theoremstyle{definition}
\newtheorem{definition}{Definition}[section]
\newtheorem{remark}{Remark}[section]

\newtheorem{question}{Question}[section]

\title{$C_0$-semigroups of $m$-isometries on Hilbert spaces}
\author{T. Berm\'{u}dez,
A. Bonilla
 and H. Zaway
\thanks{The first and  second  authors were supported by MINECO and FEDER,
Project MTM2016-75963-P. The third author was supported in part by Departamento de Análisis Matemático of  Universidad de La Laguna and by a grant of Université de Gabès, UNG 933989527.}}

\begin{document}

\maketitle
\begin{abstract}
Let  $\{T(t)\}_{t\ge 0}$ be  a $C_0$-semigroup on a  separable Hilbert space $H$. We characterize that  $T(t)$ is an $m$-isometry for every $t$ in terms that the mapping $t\in \Bbb R^+ \rightarrow \|T(t)x\|^2$ is a polynomial of degree less  than $m$ for each $x\in H$. This fact is used  to study  $m$-isometric right translation semigroup on weighted $L^p$-spaces.
We  characterize the above property  in terms of  conditions on the infinitesimal  generator operator or in terms of  the cogenerator operator  of $\{ T(t)\}_{t\geq 0}$.
Moreover, we prove that a  non-unitary $2$-isometry on a Hilbert space satisfying the kernel condition, that is,
$$
T^*T(KerT^*)\subset KerT^*\;,
$$
then $T$ can be embedded into a $C_0$-semigroup if and only if $dim (KerT^*)=\infty$.

\end{abstract}

\section{Introduction}

Let $H$ be a complex Hilbert space and $B(H)$ denote the $C^*$-algebra of all bounded linear operators on $H$.

A one parameter family $\{ T(t)\}_{t\geq 0}$ of bounded linear operators from $H$ into $H$ is a \emph{$C_0$-semigroup} if:
\begin{enumerate}
\item $T(0)=I$.
\item $T(s+t)=T(t)T(s)$ for every $t,s\geq 0$.
\item $\displaystyle \lim_{t\to 0^+}T(t)x=x$ for every $x\in H$, in the strong operator topology.
\end{enumerate}

The linear operator $A$ defined by
$$
Ax:=  \lim_{t\rightarrow 0^+}\frac{T(t)x-x}{t},
$$
for every
$$
x\in D(A):=\{ x\in H \;\;:\;\; \lim_{t\to 0^+} \frac{T(t)x-x}{t}\mbox{ exists }\}
$$
is called the \emph{infinitesimal generator} of the semigroup $\{ T(t)\}_{t\geq 0}$. It is well-known that $A$ is a closed densely defined linear operator.

If 1 is in the resolvent set of $A$, $\rho (A)$, then the Caley transform  of $A$ defined  by $V:=(A+I)(A-I)^{-1}$ is a bounded linear operator  since $V=I+2(A-I)^{-1}$, called the \emph{cogenerator}  of the $C_0$-semigroup  $\{ T(t)\}_{t\geq 0}$.

Notice that  the point 1 could be on the spectrum of $V$, $\sigma (V)$.

\ \par

For a positive integer $m$, an operator $T\in B(H)$ is called an \emph{$m$-isometry} if
$$
\sum _{k=0}^m { m\choose k} (-1)^{m-k} \|T^{k}x\|^2 =0\; ,
$$
for any $x\in H$.

We say that $T$ is a \emph{strict $m$-isometry } if $T$ is an $m$-isometry but it is not an $(m-1)$-isometry.

\ \par

\begin{remark}
{\rm  \begin{enumerate}
\item For $m\ge 2$, the strict $m$-isometries are not power bounded.  If $T$ is an $m$-isometry, then $\|T^nx\|^2 $ is a polynomial  at $n$ of degree at most $m-1$, for very $x$, \cite[Theorem~2.1]{BMNe}. In particular, $\|T^n\| =O(n)$ for  $3$-isometries and
$\|T^n\| =O(n^{\frac{1}{2}})$ for  $2$-isometries.
\item There are no strict $m$-isometries on finite dimensional spaces for even $m$. See
\cite[Proposition~1.23]{AgS}.
\item If $T$ is an $m$-isometry, then $\sigma (T)=\overline {\mathbb{D}} $ or $\sigma (T) \subseteq \partial \mathbb{D}$,  \cite[Lemma~1.2]{AgS}.
\end{enumerate}
}
\end{remark}

\ \par

The remainder of the paper is organized as follows. In Section 2, we characterize that  $T(t)$ is an $m$-isometry for every $t$ if the mapping $t\in \Bbb R^+ \rightarrow \|T(t)x\|^2$ is a polynomial of degree less  than $m$ for each $x\in H$. This fact is used in the last section in order to study  $m$-isometric right translation semigroup on weighted $L^p$-spaces.
Also, we characterize the above property  in terms of  conditions on the generator operator or in terms that the cogenerator operator be  an $m$-isometry. Moreover, we obtain that $\{ T(t)\}_{t\geq 0}$  is an $m$-isometry for all $t>0$ if and only if $T(t)$ is an $m$-isometry, for all $t\in [0,t_1]$ with $t_1>0$ or on  $[t_1,t_2]$ with $0<t_1<t_2$.

Section 3 is devoted to embed $m$-isometries into $C_0$-semigroups. That is, given  an $m$-isometry $T$ finds a $C_0$-semigroup  $\{ T(t)\}_{t\geq 0}$  such that $T(1)=T$. Using the model for $2$-isometries  we conclude that a non-unitary $2$-isometry  on a Hilbert space such that satisfies the \emph{kernel condition}, that is
$$
T^*T(Ker (T^*))\subset Ker (T^*)\;,
$$
can be embedded into a $C_0$-semigroup   if and only if  $\dim (Ker T^*)=\infty $.

Finally, in Section 4, we obtain a characterization of the right  translation $C_0$-semigroup  on some weighted space, to be a semigroup of $m$-isometries for all $t>0$.

\section{$C_0$-semigroups of $m$-isometries}

Recall that any $C_0$-semigroup $\{ T(t)\}_{t\geq 0}$  have some real quantities associated as:
\begin{enumerate}
\item The \emph{spectral bound}, $s(A)$, by setting
$$
s(A):=\sup\{Re\lambda\;\; :\;\; \lambda \in \sigma (A)\} \;.
$$
\item \emph{Growth bound}, $w_0$, given by
$$
w_0 := \inf \{ w\in \Bbb R:  \exists  M_w\ge 1 \;\;:\;\;  \|T(t)\|\le M_we^{wt}, \;\;  \forall t\ge 0 \} \;.
$$
\end{enumerate}
The above quantities are related in the following way, $s(A) \le w_0$
 and also
\begin{equation}\label{ll}
w_0=\frac{1}{t} \log r(T(t))\;,
\end{equation}
where $r(T(t))$ denotes the spectral radius of the operator $T(t)$.

The following lemma allow us to define the cogenerator of a $C_0$-semigroup  of $m$-isometries.

\begin{lemma}
Let $\{T(t)\}_{t\ge 0}$ be a $C_0$-semigroup on a separable Hilbert space $H$ consisting of $m$-isometries and $A$ its generator. Then $1\in  \rho(A)$
and therefore the cogenerator $V$ of $\{T(t)\}_{t\ge 0}$ is well-defined.
\end{lemma}
\begin{proof}
If $T(t)$ is an $m$-isometry for all $t$, then the spectral radius of $T(t)$, $r(T(t))$  is 1 for all $t$. Using equality (\ref{ll}), then  $s(A)\le w_0=0$ and then $1\in  \rho(A)$.
\end{proof}

 The following combinatorial result will be necessary for the proof of Theorem~\ref{Thc}.

Denote  ${m \choose k}=0$ if $m<k$ or $k<0$.

\begin{lemma}\label{p}
 Let $m$ be a positive  integer and $p,\, q$ be integers such that $0\leq p,q\leq m$.
 \begin{enumerate}
 \item If $p+q\neq m$, then
 $$
 \sum_{k=0}^m {m \choose k} (-1)^{m-k} \left\{ \sum_{i=0}^p { m-k \choose i} {k \choose p-i}(-1)^i\right\}
 $$
 $$
 \left\{ \sum_{j=0}^q{m-k \choose j} {k \choose q-j} (-1)^j\right\}=0
 $$
  \item If $p+q=m$, then
$$
 \sum_{k=0}^m{m \choose k} \left\{ \sum_{i=0}^q {m-k \choose i} {k \choose q-i} (-1)^i\right\}^2=
  2^m{ m \choose q} = 2^m{ m \choose p} \;.
 $$
 \end{enumerate}
 \end{lemma}

\begin{proof}
We define the polynomials $r$ and $s$  by
 $r(x,y):= 2^m (x+y)^m$ and
 $$
 s(x,y):= ((x+1)(y+1)-(x-1)(y-1))^m\;.
 $$
 Then
\begin{eqnarray*}
r(x,y)&=& 2^m\sum_{k=0}^m {m \choose k} x^ky^{m-k}\\
s(x,y)&=& \sum_{k=0}^m {m \choose k} (-1)^{m-k} (x+1)^k(x-1)^{m-k}(y+1)^k(y-1)^{m-k}\\
&=& \sum_{k=0}^m {m \choose k} (-1)^{m-k} \displaystyle \sum_{h, \;\ell=0}^k {k \choose \ell} {k \choose h}
\displaystyle \sum_{i,\;j =0}^{m-k} {m-k \choose i} {m-k \choose j} (-1)^{i+j} x^{j+h}y^{i+\ell} \\
&=& \sum_{k=0}^m {m \choose k} (-1)^{m-k} \left\{ \sum_{i=0}^p { m-k \choose m-k-i} {k \choose k+i-p}(-1)^i\right\}\\
 & & \left\{ \sum_{j=0}^q{m-k \choose m-k-j} {k \choose k+j-q} (-1)^j\right\}x^{j+h}y^{i+\ell}\\
&=& \sum_{k=0}^m {m \choose k} (-1)^{m-k} \left\{ \sum_{i=0}^p { m-k \choose i} {k \choose p-i}(-1)^i\right\}\\
& & \left\{ \sum_{j=0}^q{m-k \choose j} {k \choose q-j} (-1)^j\right\}x^{j+h}y^{i+\ell}
 \;.
\end{eqnarray*}
Denote $\widehat{x^py^q}^f$ the coefficient  of the power  $x^py^q$ on the polynomial  $f$. Since $s(x,y)=r(x,y)$, then
$\widehat{x^py^q}^s=\widehat{x^py^q}^r$ for every $p$ and $q$. So, if $p+q\neq m$, then $\widehat{x^py^q}^s=\widehat{x^py^q}^r=0$. If $p+q= m$, then
$$
\widehat{x^py^q}^s=\widehat{x^py^q}^r=\widehat{x^py^{m-p}}^r=2^m{m \choose p}=2^m {m \choose m-p}=2^m{m \choose q} \;.
$$
By other hand, it is straightforward to verify that, if $p+q=m$, then
\begin{eqnarray*}
\widehat{ x^py^q}^s&=& \sum_{k=0}^m {m \choose k} \left\{ \sum _{i=0}^q{m-k \choose m-k-i} {k \choose k+i-q} (-1)^i\right\}^2\\
& =& \sum_{k=0}^m {m \choose k} \left\{ \sum _{i=0}^q{m-k \choose i} {k \choose q-i} (-1)^i\right\}^2 \;.
\end{eqnarray*}
This complete the proof.
\end{proof}

Given a $C_0$-semigroup $\{ T(t)\}_{t\geq 0}$, we have two different operators associated to $\{ T(t))\}_{t\geq0}$, the infinitesimal generator $A$ and the cogenerator $V$, if $1\in \rho (A)$. This two operators will be useful  in the next result, where we obtain a natural generalization to  $C_0$-semigroup of $m$-isometries of \cite[Proposition~2.2]{Eva}. See also \cite[Proposition~2.6]{Partington} and \cite[Theorem~2]{Rydhe}.

\begin{theorem}\label{Thc}
Let $\{T(t)\}_{t\ge 0}$ be a $C_0$-semigroup on a  separable Hilbert space $H$. Then the following conditions are equivalent:
\begin{enumerate}
\item[(i)] $T(t)$ is an $m$-isometry for every $t$.
\item[(ii)] The mapping $t\in \Bbb R^+ \rightarrow \|T(t)x\|^2$ is a polynomial of degree less  than $m$ for each $x\in H$.
\item[(iii)] The operator inequality
$$
\sum _{k=0}^m {m \choose k} \langle A^{m-k}x, A^{k}x \rangle=0\;,
$$
for any  $x\in D(A^m)$.
\item[(iv)] The cogenerator $V$ of  $\{T(t)\}_{t\ge 0}$ exists and is an $m$-isometry.
\end{enumerate}
\end{theorem}
\begin{proof}
$(i)\Leftrightarrow (ii)$ If   $T(t)$ is an $m$-isometry for every $t$, then
\begin{equation}\label{ecm}
\sum _{k=0}^m { m\choose k}(-1)^{m-k} \|T(t+k\tau)x\|^2 =0\; ,
\end{equation}
for  every $t,\; \tau > 0$ and $x\in H$. From the  assumption on the semigroup, it is clear that the function
 $t\in \Bbb R^+ \rightarrow f(t):= \|T(t)x\|^2$ is continuous. By \cite[Theorem~13.7]{Ku1}, the function $f(t)$ is  a polynomial of degree less than $m$ for each $x\in H$.

Conversely, if the  mapping $t\in \Bbb R^+ \rightarrow \|T(t)x\|^2$ is a polynomial of degree less  than $m$ for each $x\in H$, then $T(t)$ is an $m$-isometry for every $t$, \cite[page~271]{Ku1}.

$(ii)\Leftrightarrow (iii)$ Let  $y\in D(A^m)$. The function $t\in \Bbb R^+ \rightarrow \|T(t)y\|^2$ has $mth$-derivative and it is given by
\begin{equation}\label{eca}
\sum _{k=0}^m { m\choose k} \langle A^{m-k}T(t)y, A^{k}T(t)y \rangle \;.
\end{equation}
By {\it (ii)} and (\ref{eca}) at $t=0$ we have that
$$
\sum _{k=0}^m { m\choose k} \langle A^{m-k}x, A^{k}x \rangle  =0 \;,
$$
for any $x\in D(A^m)$.

Conversely, if (\ref{eca}) holds on $D(A^m)$, then the $mth$-derivative of the function $t\in \mathbb{R}^+\longrightarrow \| T(t)x\|^2$ agree with
$$
\sum _{k=0}^m { m\choose k} \langle A^{m-k}T(t)x, A^{k}T(t)x \rangle \;,
$$ for every $t>0$ and $x\in D(A^m)$. Since $D(A^m)$ is dense by \cite[Theorem~2.7]{pazy}, we obtain the result.

$(iii)\Leftrightarrow (iv)$ It will be sufficient to prove that
\begin{equation}\label{bb}
2^m\sum_{k=0}^m{m \choose k} \langle A^{m-k} x,A^kx\rangle= \sum_{k=0}^m {m \choose k} (-1)^{m-k} \langle (A+I)^k(A-I)^{m-k}x,(A+I)^k(A-I)^{m-k}x \rangle
\end{equation}
for all $x\in D(A^m)$, since (\ref{bb}) is equivalent to
\begin{eqnarray*}
& & 2^m\sum_{k=0}^m \langle A^{m-k}(A-I)^{-m}y, A^k(A-I)^{-m}y\rangle\\
& & =
\sum_{k=0}^m {m \choose k} (-1)^{m-k} \langle (A+I)^k(A-I)^{-k}y, (A+I)^k(A-I)^{-k} y\rangle\\
& & =
\sum_{k=0}^m {m \choose k} (-1)^{m-k} \| V^ky\|^2 \;,
\end{eqnarray*}
for all $y\in R(A-I)^m$.

Note that the second part of  equality (\ref{bb}) is given by
\begin{equation}\label{bbv}
\sum_{k=0}^m {m \choose k} (-1)^{m-k} \sum_{\ell ,\; h=0}^k {k \choose \ell}  {k\choose h} \sum_{i,\; j=0}^{m-k} {m-k \choose i} {m-k \choose j} (-1)^{i+j} \langle A^{\ell+i}x, A^{h+j}x\rangle \;.
\end{equation}
We denote $\widehat{A}_{i,j} $ the numerical coefficient of $\langle A^ix, A^jx\rangle$. It is straightforward to verify that if $p+q=m$, then
\begin{eqnarray*}
\widehat{A}_{q,m-q}&=& \widehat{A}_{m-q,q}\\
&=& \sum_{k=0}^m {m \choose k} \left\{ \sum_{i=0}^q {m-k \choose m-k-i} {k \choose k+i-q} (-1)^i \right\}^2 \\
&=& \sum_{k=0}^m {m \choose k} \left\{ \sum_{i=0}^q {m-k \choose i} {k \choose q-i} (-1)^i \right\}^2 =
2^m{m \choose q} \;,
\end{eqnarray*}
where the last equality is obtained by part (2) of Lemma \ref{p}.

If $p+q\neq m$, then
\begin{eqnarray*}
\widehat{A}_{p,q}&=& \widehat{A}_{m-p,m-q}\\
&=& \sum_{k=0}^m {m \choose k} (-1)^{m-k} \left\{ \sum_{i=0}^p { m-k \choose m-k-i} {k \choose k+i-p}(-1)^i\right\}\\
& & \left\{ \sum_{j=0}^q{m-k \choose m-k-j} {k \choose k+j-q} (-1)^j\right\}
\end{eqnarray*}
\begin{eqnarray*}
&=&\sum_{k=0}^m {m \choose k} (-1)^{m-k} \left\{ \sum_{i=0}^p { m-k \choose i} {k \choose p-i}(-1)^i\right\}\left\{ \sum_{j=0}^q{m-k \choose j} {k \choose q-j} (-1)^j\right\}\\
&=&0 \;,
\end{eqnarray*}
where the last equality is obtained by part (1) of Lemma \ref{p}. So we get the result.
\end{proof}

\begin{corollary}\label{Thcc}
Let $\{T(t)\}_{t\ge 0}$ be a $C_0$-semigroup on a  separable Hilbert space $H$. Then $T(t)$ is a strict $m$-isometry for every $t> 0$ if and only if the cogenerator $V$ of  $\{T(t)\}_{t\ge 0}$ is a strict $m$-isometry.
\end{corollary}

In the following corollary, we give  an  example of $m$-isometric semigroup.

\begin{corollary}
Let $Q$ be a nilpotent operator of order $n$ on a  separable Hilbert space. Then the $C_0$-semigroup generated by $Q$ is a strict  $(2n-1)$-isometric semigroup.
\end{corollary}
\begin{proof}
Since $Q$ is the generator of $\{T(t)\}_{t\ge 0}$ and $1\in \rho (Q)$, then the cogenerator is well-defined and given by
$$
V:=(Q+I)(Q-I)^{-1}\;.
$$
Thus $V=-(Q+I)(I+Q+ \cdots + Q^{n-1})= -I-2Q(I+Q+ \cdots + Q^{n-2})$, that is, the sum of an isometry and a nilpotent operator of order $n$. By \cite[Theorem~2.2]{BMNo}, the cogenerator is a strict  $(2n-1)$-isometry.
 Then  $\{T(t)\}_{t\ge 0}$ is a strict  $(2n-1)$-isometric semigroup by Corollary \ref{Thcc}.
\end{proof}

For a positive integer $m$, a linear closed  operator $A$ defined on a dense set $D(A)\subset H$ is called an \emph{$m$-symmetry} if
$$
\sum _{k=0}^m { m\choose k} (-1)^{m-k}\langle A^{m-k}x,A^kx\rangle =0\; ,
$$
for all $x\in D(A^m)$.
We say that $A$ is a \emph{strict $m$-symmetry } if $A$ is an $m$-symmetry but it is not an $(m-1)$-symmetry. There is not  strict $m$-symmetry bounded  for even  $m$. See \cite[Page~7]{Ag0}.

In the next result, we present the connection between  the condition (iii) of Theorem \ref{Thc} and  the $m$-symmetric operators.

\begin{corollary}
Let $A$ be an   \emph{$m$-symmetry}    on a  separable Hilbert space. Then the $C_0$-semigroup generated by $iA$ is an $m$-isometric semigroup.
\end{corollary}
\begin{proof}
If $A$ is  an   $m$-symmetry,  then
$$
\sum _{k=0}^m {m \choose k} \langle (iA)^{m-k}x, (iA)^{k}x\rangle= (-i)^m \sum _{k=0}^m (-1)^{m-k}{m \choose k} \langle A^{m-k}x, A^kx\rangle=0\;,
$$
for all $x\in D(A^m)$.
Thus the $C_0$-semigroup generated by $iA$ is an $m$-isometric semigroup. The result is completed by  Theorem \ref{Thc}.
\end{proof}

Let $w$ be a weighted function defined on $\mathbb{T}$. It is defined
$$
L^2_{w}(\mathbb{T}):=\{ f:\mathbb{T}\rightarrow \mathbb{C}\;\;:\;\;\; \int _{\mathbb{T}}|f(z)|^2 w(z)dz<\infty\}\;.
$$

 Let $\{T(t)\}_{t\ge 0}$ be a $C_0$-group defined on $L^2_w(\Bbb T)$ with non-constant weighted function $w$ by $ T(t)f(z):=f(e^{it}z)$. Then $T(t)$ is an isometry if and only if $t$ is a multiple of $2\pi$. See \cite[page~8]{CP}.

 \begin{proposition}
 Let $\{T(t)\}_{t\in \mathbb{R}}$ be a $C_0$-group on a Hilbert space $H$. Then the following conditions  are equivalent:
\begin{enumerate}
\item[(i)] $T(t)$ is an $m$-isometry for every $t$.
\item[(ii)] $T(t)$ is an $m$-isometry for $t_1$ and $t_2$ where $\frac{t_1}{t_2}$ is irrational.
\end{enumerate}
\end{proposition}
 \begin{proof}
For  every $t$, we have that
$$
\sum _{k=0}^m (-1)^{m-k}{m\choose k}\|T(t+kt_i)x\|^2 =0\; \mbox{ for } i=1,2\;.
$$
Montel's Theorem (\cite[Theorem~13.5]{Ku1} \& \cite[Theorem~1.1]{PR}) gives that $\|T(t)x\|^2$ is a polynomial of degree less than $m$ for each $x\in H$, since $\frac{t_1}{t_2}$ is irrational. Thus by Theorem \ref{Thc} we have that  $T(t)$ is an $m$-isometry for every $t$.
\end{proof}

 If $T$ is an $m$-isometry, then any power $T^r$ is also an $m$-isometry, \cite[Theorem~2.3]{J}.
 In general, the converse is not true. However, if $T^r$ and $T^{r+1}$ are
$m$-isometries for a positive integer $r$, then $T$ is an $m$-isometry (see, \cite[Corollary~3.7]{BDM}). The stability of the class of $m$-isometry with powers is fundamental to give necessary and sufficient conditions for a $C_0$-semigroup $\{ T(t)\}_{t\geq 0}$ to be $T(t)$ an $m$-isometry for all $t\geq 0$.

\begin{theorem}\label{Thi}
Let $\{T(t)\}_{t\ge 0}$ be a $C_0$-semigroup on a Hilbert space $H$. Then the following properties are equivalent:
\begin{enumerate}
\item[(i)] $T(t)$ is an $m$-isometry for every $t\geq     0$.
\item [(ii)] $T(t)$ is an $m$-isometry for every $t\in [0,t_1]$ for some $t_1>0$.
\item[(iii)] $T(t)$ is an $m$-isometry on an interval of the form $[t_1, t_2]$ with $t_1< t_2$.
\end{enumerate}
\end{theorem}
\begin{proof}
The implications $(i)\Rightarrow(ii)$ and $(ii)\Rightarrow(iii)$ are clear.

$(ii)\Rightarrow (i)$.  Fixed $t>0$, there exist $n\in \NN$ and $t'\in [0,t_1]$ such that $t=nt'$. Since $T(t')$ is an $m$-isometry, then any power of $T(t')$ is an $m$-isometry by \cite[Theorem~3.1]{BDM}. So, $T(t)$ is an $m$-isometry.

$(iii)\Rightarrow (i)$. Let us prove that $T(t)$ is an $m$-isometry for every $t\in (0,\frac{t_2 -t_1}{4}]$.

Choose $ k:= [\frac{t_1}{t}] +1$. Then $kt$ and $(k+1)t$ belong to $(t_1, t_2)$. Thus  $T(t)^{k}$ and $T(t)^{k+1}$ are $m$-isometries.
Hence $T(t)$ is an $m$-isometry by \cite[Corollary~3.7]{BDM}.
\end{proof}

\section{Embedding $m$-isometries  into $C_0$-semigroups}

We  are interested  to study: When can we embed an $m$-isometry into a continuous $C_0$-semigroup?, that is, given an $m$-isometry $T$, to find a $C_0$-semigroup $\{ T(t)\}_{t\geq 0}$ such that $T(1)=T$.

Recall that an isometry $T$ on a Hilbert space can be embedded   into a
$C_0$-semigroup if and only if $T$ is unitary or
$codim(R(T)) =\infty$, where $R(T)$ denotes the range of $T$. In this case, it is also possible to embed $T$  into an isometric $C_0$-semigroup \cite[Theorem~V.1.19]{E}.

Note that, if $T$ can be embedded into $C_0$-semigroup, then $dim (Ker T)$ and \linebreak $dim (Ker T^*)$ are zero or infinite \cite[Theorem~V.1.7]{E}.

\begin{proposition}
\begin{enumerate}
\item[(i)] All $m$-isometries in finite dimensional space are  embeddable on a $C_0$-group.

\item[(ii)] All  normal $m$-isometries are   embeddable.
\item[(iii)] The weighted forward shift $m$-isometries are not embeddable.
\end{enumerate}
\end{proposition}

\begin{proof}

{\it (i)} On a finite-dimensional space an operator can be
 embeddable if and only if its spectrum does not contain 0 (see, \cite[page~166]{E}). Moreover, on finite-dimensional space the spectrum of $m$-isometries is contained in the  unit circle. Hence any  $m$-isometry on finite dimensional space is embeddable in a group.

{\it (ii)} If $T$ is a normal $m$-isometry, then $T$ is invertible and  by \cite[Theorem~V.1.14]{E} $T$ is embeddable.

{\it (iii)} Assume that $T$ is a weighted forward shift $m$-isometry. Then $dim Ker(T^*)=1$, (\cite{Al} \& \cite{BMNe}).
 \end{proof}

\bigskip

Let $M_z$ be the multiplication operator on the Dirichlet space $D(\mu)$ for some finite non-negative Borel measure on $\Bbb T$ defined by
$$
D(\mu):=\{ f:\mathbb{D} \rightarrow \mathbb{C} \; \mbox{ analytic } : \;\; \int_{\mathbb{D}}|f'(z)|^2\varphi_{\mu} (z) dA(z)<\infty\} \;,
$$
where $A$ denotes the normalized Lebesgue area measure on $\mathbb{D}$ and $\varphi_{\mu}$ is defined by $$
\varphi _{\mu}(z):= \frac{1}{2\pi}\int _{[0,2\pi)}\frac{1-|z|^2}{|e^{it}-z|^2} d\mu (t)\;,
$$
for $z\in \mathbb{D}$.

\begin{proposition}
 $M_z$ on $D(\mu )$ can not be embedded into $C_0$-semigroup.
\end{proposition}

\begin{proof}

By  Richter's Theorem~\cite{Richter}, $M_z$ is an analytic $2$-isometry with  $dim (Ker T^*) =1$, then $M_z$ can not be embedded into $C_0$-semigroup.
\end{proof}

\ \par

Let $Y$ be an infinite dimensional Hilbert space. Denote by  $\ell^2_Y$  the Hilbert space of all vector sequences $(h_n)_{n=1}^{\infty}$ such that $\sum_{n\geq 1} \|h_n\|^2 <\infty$ with the standard inner product.
If $( W_n)_{n=1}^{\infty} \subset B(Y)$ is an uniformly bounded sequence of operators, then the operator $S_W \in B( \ell^2_Y)$ defined by
$$
S_W( h_1, h_2, \cdots ) := (0, W_1h_1,  W_2h_2, \cdots )\;,
$$
for any  $(h_1, h_2, \cdots )\in  \ell^2_Y$,
is called the   \emph{operator valued unilateral forward weighted shifts} with weights $W:=(W_n)_{n\geq 1}$.

Following some ideas of \cite[Proposition~V.1.18]{E},  we obtain the next result.

\begin{lemma}\label{vufs}
Let $S_W$ be the  valued unilateral forward weighted shift operator with weights $W=(W_n)_{n\geq 1}$ on $\ell ^2_Y$ where  $Y$ be an infinite dimensional Hilbert space. Then $S_W$ can be embedded into  $C_0$-semigroup.
\end{lemma}
\begin{proof}
The operator  $S_W$ is unitarily equivalent to the   valued unilateral forward weighted shift operator on $\ell ^2_{L^2([0,1),Y)}$. Moreover,  $\ell ^2_{L^2([0,1),Y)}$ can be identified with $L^2(\Bbb R^+,Y)$ by
$$
(f_1,f_2 \cdots ) \rightarrow ( s\rightarrow f_n(s-n), \;\; s \in [n, n+1))\;.
$$

For $0<t\le 1$, define the following family of operators on $L^2(\mathbb{R}^+,Y)$ by
$$
(T(t)f)(s):=\left\{
\begin{array}{ll}
0 & s<t\\[0.7pc]
f(s-t) & n-1+t\leq s<n\\[0.7pc]
W_nf(s-t) & n\leq s<n+t \;,
\end{array}
\right.
$$
that is,
$$
(T(t)f)(s)=\left\{
\begin{array}{ll}
0 & s<t\\[0.7pc]
\displaystyle \sum_{n\geq 1} \left( f(s-t)\chi _{[n-1+t,n)}(s)+W_n f(s-t)\chi_{[n,n+t)}(s)\right)  & s\geq t \;.
\end{array}
\right.
$$
In particular, for $t=1$, we have that
$$
(T(1)f)(s)
=\left\{
\begin{array}{ll}
0 & s<1\\[0.7pc]
\displaystyle \sum_{n\geq 1} W_nf(s-1)\chi _{[n,n+1)}(s) & s\geq 1 \;.
\end{array}
\right.
$$
Hence $T(1)$ is unitarily equivalent to $S_W$.

For $t>1$,  define $T(t):=T^{[t]}(1)T(t-[t])$, where $[t]$ denotes the  greatest integer less than or equal to $t$.

Let us prove that $T(t)$ is a $C_0$-semigroup. Given any $f\in L^2(\mathbb{R}^+,Y)$, we have that
$$
\lim_{t\to 0^+}T(t)f(s)=\sum_{n\geq 1} \chi_{[n-1,n)}f(s)=f(s)\;,
$$
in the strong topology of $L^2(\mathbb{R}^+,Y)$.

Let $t,t'\in [0,1)$. Then  $T(t)T(t')f(s)=T(t)\tilde{f}(s)$, where
$$
\tilde{f}(s):=T(t')f(s)=\left\{
\begin{array}{ll}
0 & s<t'\\[0.7pc]
\displaystyle \sum_{n\geq 1} (\chi_{[n-1+t',n)}(s)+W_n\chi_{[n,n+t')}(s))f(s-t') & s\geq t'\;.
\end{array}
 \right.
$$
Then
$$
T(t)T(t')f(s) 
$$
\begin{eqnarray*}
& =& \left\{
\begin{array}{ll}
0 & s<t+t'\\[0.7pc]
\displaystyle \sum_{m\geq 1} \Biggl( 
\displaystyle \sum _{n\geq 1} \left( \chi_{[n-1+t',n)}(s-t)+W_n\chi_{[n,n+t')}(s-t)
\right)
\chi_{[m-1+t,m)}(s)+  &\\
W_m \displaystyle  \sum_{n\geq 1} \left( \chi _{[n-1+t',n)}(s-t)+W_n\chi_{[n,n+t')}(s-t)\right) \\
\chi_{[m,m+t)}(s)\Biggr)f(s-t'-t) 
 & s\geq t+t'
\end{array}
 \right.
\end{eqnarray*}
\begin{eqnarray}\label{v}
& =& \left\{
\begin{array}{ll}
0 & s<t+t'\\[0.7pc]
\displaystyle\sum_{m\geq 1} \Biggl\{  \displaystyle \sum_{n\geq 1}\left(\chi_{[n-1+t'+t,n+t)}(s)+W_n\chi_{[n+t,n+t'+t)}(s)\right) \chi_{[m-1+t,m)}(s) + & \\[0.7pc]
 W_m\displaystyle  \sum_{n\geq 1} \left( \chi_{[n-1+t'+t,n+t)}(s)+W_n\chi_{[n+t,n+t'+t)}(s)\right)\\[0.7pc]
\chi_{[m,m+t)}(s)\Biggr\}f(s-t'-t) & s\geq t+t'
\end{array}
 \right.
 \end{eqnarray}

If $t'+t<1$, then (\ref{v}) is given by
\begin{eqnarray*}
T(t)T(t')f(s)& =& \left\{
\begin{array}{ll}
0 & s<t'+t\\
\displaystyle \sum_{n\geq 1} \left( \chi_{[n-1+t'+t,n)}(s)+W_n\chi _{[n,n+t'+t)}(s)\right) f(s-t'-t)  & s\geq t'+t
\end{array}
\right. \\
&=& T(t+t')f(s)\;.
\end{eqnarray*}

Denote $t'':=t'+t-[t'+t]$. If $t'+t>1$, then $t''=t'+t-1$ and
\begin{eqnarray*}
T(t+t')f(s)& =& T^{[t+t']}(1)(T(t+t'-[t+t']))f(s)\\[0.7pc]
&=& T(1)T(t+t'-1)f(s)=T(1)T(t'')f(s)\\[0.7pc]
&=& \left\{
\begin{array}{ll}
0 & s<t''\\[0.7pc]
T(1) \displaystyle \sum_{n\geq 1} (\chi_{[n-1+t'',n)}(s)+W_n\chi_{[n,n+t'')}(s)) f(s-t'')& s\geq t''
\end{array}
 \right.\\[0.7pc]
 &=& \left\{
\begin{array}{ll}
0 & s-t''<1\\[0.7pc]
\displaystyle \sum_{m\geq 1} W_m \sum_{n\geq 1} \left( \chi _{[n-1+t'',n)}(s-1)+ W_n\chi_{[n,n+t'')}(s-1) \right)\\
 \chi _{[m, m+1)}(s) f(s-t''-1) & s-t''\geq 1
\end{array}
\right. \\[0.7pc]
&=& \left\{
\begin{array}{ll}
0 & s-t''<1\\[0.7pc]
\displaystyle \sum_{m\geq 1} W_m \displaystyle \sum  _{n\geq 1} \left( \chi _{[n+t'',n+1) }(s)+ W_n\chi _{[n+1,n+t''+1)}(s) \right)\\[0.7pc]
\chi _{[m,m+1)}(s) f(s-t''-1) & s\geq 1+t''
\end{array}
\right.
\end{eqnarray*}
\begin{eqnarray*}
&=&
\left\{
\begin{array}{ll}
0 & s<t+t'\\[0.7pc]
\displaystyle  \sum_{n\geq 1} \left( W_n\chi_{[n+t'',n+1)}(s)+W_{n+1}W_n\chi_{[n+1,n+t''+1)}(s)\right)f(s-t'-t) & s\geq t+t' \;.
\end{array}
\right.
\end{eqnarray*}

By other hand, by (\ref{v}) we have that
$$
T(t)T(t')f(s)=
$$
\begin{eqnarray*}
\left\{
\begin{array}{ll}
0 & s<t+t'\\[0.7pc]
 \displaystyle  \sum_{n\geq 1} \left( W_n\chi_{[n+t'+t-1,n+1)}(s)+W_{n+1}W_n\chi_{[n+1,n+t'+t)}(s)\right)f(s-t'-t) & s\geq t+t'
\end{array}
 \right.
 \end{eqnarray*}
This completes the proof.
\end{proof}

We say that an operator $T\in B(H)$ satisfies the \emph{kernel condition} if
$$
T^*T(Ker T^*)\subset Ker T^*\;.
$$

\begin{corollary}
A  non-unitary $2$-isometry on a Hilbert space satisfying the kernel condition can be embedded into $C_0$-semigroup if and only if $dim (KerT^*)=\infty$.
\end{corollary}

\begin{proof}
If $T$ is a non-unitary $2$-isometry on a Hilbert space satisfying the kernel condition as consequence of  \cite[Theorem~3.8]{ACJS}, we obtain that  $T\cong U \oplus W$ with $U$ unitary and $W$ a operator valued unilateral forward weighted shifts operator  in $\ell_{M}^2$ with $dim M = dim (KerT^*)$. Thus by \cite[Theorem~V.1.19]{E} and Lemma \ref{vufs}, $T$ can be embedded into $C_0$-semigroup.
\end{proof}

By the Wold Decomposition Theorem for $2$-isometries (see \cite{Olo}), all $2$-isometries can be decomposed as directed sum of  unitary operator and an analytic $2$-isometries.

\begin{question} An analytic $2$-isometry on a Hilbert space can be embedded into $C_0$-semigroup if and only if $dim (KerT^*)=\infty$?
\end{question}

\section{Translation semigroups of $m$-isometries}

In this section,  we discuss examples of semigroups of $m$-isometries.

\begin{definition} By a \emph{right admissible weighted function} in $(0,\infty)$, we mean
a measurable  function $\rho:(0,\infty)\rightarrow \Bbb{R}$
satisfying the following conditions:
\begin{enumerate}
\item $\rho(\tau)>0$ for all $\tau \in (0,\infty)$,

\item there exist constants $M\geq 1$ and $\omega\in \Bbb{R}$
such that $ \rho(t+\tau)\leq Me^{\omega t}\rho(\tau)$ holds for all $\tau\in (0,\infty)$
and all $t >0$.
\end{enumerate}
\end{definition}

For a right admissible weighted function, we define the \emph{weighted space} $L^2(\Bbb R^+, \rho)$ of measurable  functions $f:\Bbb R^+\rightarrow \Bbb C$ such that
$$
\|f\|_{L^2(\Bbb R^+,\; \rho)} = \int_0^{\infty}|f(s)|^2\rho(s)ds<\infty \;.
$$
Then  the \emph{right translation semigroup} $(S(t))_{t\ge 0}$  given for $t \ge 0$ and $f \in L^2(\Bbb R^+, \rho)$ by
$$
(S(t)f)(s):=\left\{\begin{array}{ccl}
   0 &\mbox{ if }  s\le t \\[0.7pc]
   f(s-t)&\mbox{ if }  s> t \;,
  \end{array}\right.
$$
is  a strongly   continuous semigroup and  straightforward computation shows that for $s,t \ge 0$ and $f \in L^2(\Bbb R^+, \rho)$
$$
(S^*(t)f)(s)= \frac{\rho(s+t)}{\rho(s)}f(s+t)\;.
$$

\begin{theorem}\label{Ths}
Let $\{S(t)\}_{t\ge 0}$ be the  right translation  $C_0$-semigroup on $L^2(\Bbb R^+, \rho) $ with $\rho$ a continuous function. The semigroup  $\{S(t)\}_{t\ge 0}$ is an $m$-isometry for every $t> 0$ if and only if $\rho(s)$ is a  polynomial of degree less than $m$.
\end{theorem}
\begin{proof}
By definition,  $S(t)$ is an $m$-isometry for every $t> 0$ if
 $$
 \sum _{k=0}^m { m\choose k}(-1)^{m-k} \|S^{k}(t)f\|^2=0\;,
$$
for all $t\geq  0$  and $f \in L^2(\Bbb R^+, \rho) $.
 That is,
\begin{equation}\label{ecn}
\int _0^{\infty}\left( \sum _{k=0}^m { m\choose k} (-1)^{m-k}\frac{\rho(s+kt)}{\rho(s)}\right) |f(s)|^2\rho(s)ds=0\;,
\end{equation}
for all $t\geq  0$  and $ f \in L^2(\Bbb R^+, \rho)$.
Fixed $t\geq 0$ , define
$$
g(s):= \sum_{k=0}^m {m \choose k} (-1)^{m-k} \frac{\rho (s+kt)}{\rho (s)} \;.
$$
If $g(s)\neq 0$, we can suppose without lost of generality that, $g(s)>0$, for some $s\geq 0$. Then by continuity of $\rho $ and $\rho (\tau) >0 $ for all $\tau \geq 0$,  we obtain that there exits an interval $I\subset \mathbb{R}^+$, with finite measure,  such that there exists $M>0$ such that $g(s)>M$ for all $s\in I$.

Let $f_1(s):= \frac{1}{\sqrt{\rho (s)}}\chi_I (s) \in L^2(\mathbb{R}^+,\rho )$. Then by (\ref{ecn}) we have that
$$
0=\int _0^\infty g(s)|f_1(s)|^2\rho (s) ds=\int _Ig(s) ds>M\mu (I)\;,
$$
which it is an absurd. So,
$$
\sum_{k=0}^m{m\choose k} (-1)^{m-k} \frac{\rho (s+kt)}{\rho (s)} =0 \;,
$$
for all $s\geq 0$  and $t\geq 0$. Then by \cite[Theorem~13.5]{Ku1}, the function $\rho $ is a polynomial of degree less than $m$.
\end{proof}

\begin{corollary}
Let $\{S(t)\}_{t\ge 0}$ be the right translation  $C_0$-semigroup on $L^2(\Bbb R^+, \rho) $ with $\rho$  a continuous function. Then $\{S(t)\}_{t\ge 0}$ is a strict $m$-isometry for every $t> 0$ if and only if $\rho(s)$ is a polynomial of degree $m-1$.
\end{corollary}

Consider the \emph{right  weighted translation $C_0$-semigroup}, $S_{\rho}$,  defined on $L^2(\mathbb{R}^+)$ as
$$
(S_{\rho}(t)f)(s):=\left\{\begin{array}{ccl}
  0 &\mbox{ if }  s\le t \\[0.7pc]
  {\displaystyle \frac{\rho(s)}{\rho(s-t)}} f(s-t) &\mbox{ if }  s> t \;.
  \end{array}\right.
$$
Then $\{ S_{\rho}(t)\}_{t\geq 0}$ is a strongly   continuous semigroup if and  only if $\rho$ is a right admissible weight \cite{Embry}
and  for $s,t \ge 0$ and $f \in L^2(\Bbb R^+)$
$$
 (S_{\rho}^*(t)f)(s)= \frac{\rho(s+t)}{\rho(s)}f(s+t)\;.
$$
We now improve part (2) of  \cite[Corollary~3.3]{PS}.

\begin{theorem}
Let $\{S_{\rho}(t)\}_{t\ge 0}$ be the right  weighted translation  $C_0$-semigroup on $L^2(\Bbb R^+) $ with $\rho$ a continuous function. Then $\{S_{\rho}(t)\}_{t\ge 0}$ is an $m$-isometry for every $t> 0$ if and only if $\rho(s)^2 $ is a  polynomial of degree less than $m$.
\end{theorem}
\begin{proof}
Consider $M_{\rho}: L^2(\Bbb R^+, \rho^2) \rightarrow L^2(\Bbb R^+) $ defined by $M_{\rho}f= \rho f$. Then
$S_{\rho}(t)= M_{\rho}S(t)M_{\rho}^{-1}$. Thus $\{S_{\rho}(t)\}_{t\ge 0}$ is an $m$-isometry for every $t> 0$ on $L^2(\Bbb R^+) $ if and only if $\{S(t)\}_{t\ge 0}$ is an $m$-isometry for every $t> 0$ on $L^2(\Bbb R^+, \rho^2)$ if and only if $\rho(s)^2 $ is a polynomial of degree less than $m$.
\end{proof}

\begin{corollary}
Let $\{S_{\rho}(t)\}_{t\ge 0}$ be the right  weighted translation  $C_0$-semigroup on $L^2(\Bbb R^+) $ with $\rho$ a continuous function. Then $\{S_{\rho}(t)\}_{t\ge 0}$ is a $2$-isometry for every $t> 0$ if and only if $\rho(s)^2=as+b$  for some constants $a$ and $b$.
\end{corollary}

At continuation we characterize the weighted spaces where the adjoint of translation operator  is an  $m$-isometry.

\begin{definition} By a \emph{left admissible weighted function} in $(0,\infty)$, we mean
a measurable function $w:(0,\infty)\rightarrow \Bbb{R}$
satisfying the following conditions:
\begin{enumerate}
\item $w(\tau)>0$ for all $\tau \in (0,\infty)$,

\item there exist constants $M\geq 1$ and $\alpha \in \Bbb{R}$
such that $ w(\tau)\leq Me^{\alpha t}w(t+\tau)$ holds for all $\tau\in (0,\infty)$
and all $t >0$.
\end{enumerate}
\end{definition}

For a left admissible weighted function we define the \emph{weighted space} $L^2(\Bbb R^+, w)$ as the measurable  functions $f:\Bbb R^+\rightarrow \Bbb C$ such that
$$
\|f\|_{L^2(\Bbb R^+,\; w)} = \int_0^{\infty}|f(s)|^2 w(s)ds<\infty\;.
$$
Let $\{T(t)\}_{t\ge 0}$ be the \emph{left shift semigroup} given for $t \ge 0$ and $f\in L^2(\Bbb R^+, w)$ by
$$
(T(t)f)(s):= f(s+t)\;.
$$
Then $\{T(t)\}_{t\geq 0}$ is a strongly   continuous semigroup \cite{GEP11}. For  $f \in L^2(\Bbb R^+, w)$,
$$
(T^*(t)f)(s)=\left\{\begin{array}{ccl}
  0 &\mbox{ if }  s\le t \\[0.7pc]
   \displaystyle\frac{w(s-t)}{w(s)}f(s-t)&\mbox{ if }  s> t \;.
  \end{array}\right.
$$
\begin{theorem}
Let $\{T^*(t)\}_{t\ge 0}$ be the  adjoint of left weighted translation  $C_0$-semigroup on $L^2(\Bbb R^+, w) $ such that $w$ is a continuous function. Then $\{T^*(t)\}_{t\ge 0}$ is an $m$-isometry for every $t> 0$ if and only if $w(s) = \frac{1}{p(s)}$ for some polynomial $p$  of degree less than $m$.
\end{theorem}
\begin{proof}
$T^*(t)$ is an $m$-isometry for every $t> 0$ if and only if
$$
\sum _{k=0}^m { m\choose k}(-1)^{m-k} \|T^{*k}(t)f\|^2=0\;,
$$
for all $t> 0$  and  $f \in L^2(\Bbb R^+, w)$.
Then
\begin{eqnarray*}
0&=&\sum_{k=0}^m (-1)^{m-k}{ m\choose k} (-1)^{m-k}\int_{kt}^{\infty} \left| \frac{w(s-kt)}{w(s)}f(s-kt)\right|^2w(s)ds\\[1pc]
&=& \int_0^{\infty} \left(\sum_{k=0}^m { m\choose k} (-1)^{m-k} \frac{w(u)}{w(u+kt)}\right)|f(u)|^2w(u)du \,,
\end{eqnarray*}
for all $f\in L^2(\Bbb R^+, w)$. As in the proof of Theorem \ref{Ths} we get that $\frac{1}{w(s)} $  is a polynomial of degree less than $m$.
\end{proof}

\begin{corollary}
Let $\{T^*(t)\}_{t\ge 0}$ be the  adjoint of left weighted translation  $C_0$-semigroup on $L^2(\Bbb R^+, w) $ such that $w$ is a continuous function. Then $\{T^*(t)\}_{t\ge 0}$ is a strict  $m$-isometry for every $t> 0$ if and only if $w(s) = \frac{1}{p(s)}$ for some  polynomial $p$ of degree $m-1$.
\end{corollary}

\section*{Acknowledgements}
The authors thank to Almira  and Hernández-Abreu for calling their attention over Montel's theorems.

\small

\smallskip
\noindent T. Berm\'{u}dez\\
{\it Departamento de An\'alisis Matem\'atico, Universidad de La Laguna,
38271, La Laguna (Tenerife), Spain.}\\
{\it e-mail:} tbermude@ull.es

\smallskip
\noindent A. Bonilla\\
{\it Departamento de An\'alisis Matem\'atico, Universidad de La Laguna,
38271, La Laguna (Tenerife), Spain.}\\
{\it e-mail:} abonilla@ull.es

\smallskip
\noindent H. Zaway \\
{\it Departamento de An\'alisis Matem\'atico, Universidad de La Laguna,
38271, La Laguna (Tenerife), Spain.}\\
{\it Department of Mathematics, Faculty of Sciences, University of Gabes,
6072, Tunisia.} \\
{\it e-mail:} hajer\_zaway@live.fr

\end{document}